\documentclass[12pt,reqno]{amsart}
\usepackage{etoolbox}

\makeatletter
\patchcmd\maketitle
{\uppercasenonmath\shorttitle}
{}
{}{}
\patchcmd\maketitle
{\@nx\MakeUppercase{\the\toks@}}
{\the\toks@}
{}
{}{}
\patchcmd\@settitle{\uppercasenonmath\@title}{\Large}{}{}
\patchcmd\@setauthors
{\MakeUppercase{\authors}}
{\authors}
{}{}
\makeatother
\usepackage{amsmath,amssymb,amsthm, amsfonts}
\usepackage{color}
\usepackage{xcolor}
\usepackage{url}
\usepackage{tikz-cd}
\usepackage{combelow}
\usepackage{enumitem}
\input{mathrsfs.sty} 
\usepackage{euler}
\usepackage{palatino}
\usepackage[utf8]{inputenc}
\usepackage[T1]{fontenc}
\newtheorem{theorem}{Theorem}[section]
\newtheorem{question}{Question}[section]
 
\newtheorem{definition}{Definition}[section]

\newtheorem{corollary}{Corollary}[section]
\newtheorem{proposition}{Proposition}[section]
\newtheorem{lemma}{Lemma}[section]
\newtheorem{remark}{Remark}[section]
\newtheorem{example}{Example}[section]
\numberwithin{equation}{section}

\newcommand\norm[1]{\left\lVert#1\right\rVert}

\newcommand\skal[2]{\left\langle #1,#2\right\rangle}

\usepackage[colorlinks=true]{hyperref}
\hypersetup{urlcolor=blue, citecolor=red , linkcolor= blue}
\usepackage{cite}

\begin{document}
		\title[Structural Properties and Normality Criteria for Subclasses of Normaloid Operators]{Structural Properties and Normality Criteria for Subclasses of Normaloid Operators}
		
		
		\author[H. Stankovi\'c]{Hranislav Stankovi\'c}
		\address{Faculty of Electronic Engineering, University of Ni\v s, Aleksandra Medvedeva 4, Ni\v s, Serbia
		}
		\email{\url{hranislav.stankovic@elfak.ni.ac.rs}}
		
		\author[C. Kubrusly]{Carlos Kubrusly}
		\address{Catholic University of Rio de Janeiro, 22453-900, Rio de Janeiro, RJ, 
			Brazil
		}
		\email{\url{carlos@ele.puc-rio.br}}


		\date{\today} 
		
		\maketitle

		\begin{abstract}
			We investigate structural properties and normality criteria for certain classes of bounded linear operators on a Hilbert space. 
			We show that an operator $T$ with polar decomposition $T = U|T|$ is self-adjoint if and only if $T$ is absolute-$(p,r)$-paranormal and the partial isometry $U$ is self-adjoint. 
			Extending Ando’s Theorem, we prove that if $T$ is absolute-$(p,r)$-paranormal and $T^n$ is normal for some $n \in \mathbb{N}$, then $T$ itself is normal. 
			We further show that if $T$ is absolute-$(p,r)$-paranormal and $T^2$ is compact, then $T$ is a compact normal operator. 
			Finally, we obtain several characterizations of quasinormal partial isometries within the normaloid hierarchy.
		\end{abstract}
		
		\bigskip 
		
		\section{Introduction}\label{sec:intro}
		\bigskip
		
		Let $(\mathcal{H},\skal{\,}{\,})$ be a complex Hilbert space, and let $\mathfrak{B}(\mathcal{H})$ denote the algebra of bounded linear operators on $\mathcal{H}$. For an operator $T \in \mathfrak{B}(\mathcal{H})$, the null space and the range are denoted by $\mathcal{N}(T)$ and $\mathcal{R}(T)$, respectively. The adjoint of $T$ is denoted by $T^*$. An operator $T$ is said to be \emph{positive} ($T\geq 0$) if $\skal{Tx}{x}\geq 0$ for all $x\in\mathcal{H}$, \emph{self-adjoint} if $T=T^*$, and \emph{normal} if $T^*T = TT^*$. 
		
		An operator $P \in \mathfrak{B}(\mathcal{H})$ is called an \emph{orthogonal projection} if $P^2=P=P^*$. An operator $U\in\mathfrak{B}(\mathcal{H})$ is called a \emph{partial isometry} if $U^*U$ is an orthogonal projection, and \emph{unitary} if $U^*U=UU^*=I$. Recall that for any $T \in \mathfrak{B}(\mathcal{H})$, there exists a unique polar decomposition $T=U|T|$, where $|T|=(T^*T)^{1/2}$ and $U$ is a partial isometry satisfying $\mathcal{N}(U) = \mathcal{N}(|T|)$.
		
		\medskip 
		
		The spectral theory of normal operators is fundamental to operator theory and quantum mechanics. Consequently, considerable effort has been devoted to extending the properties of normal operators to broader classes. This has led to the study of non-normal operators defined by various operator inequalities. Some of the most prominent classes include:
		
		\begin{itemize}
			\item quasinormal operators: $T$ commutes with $T^*T$, i.e., $TT^*T = T^*T^2$;
			\item subnormal operators: $T$ is the restriction of a normal operator to an invariant subspace; i.e., there exist a Hilbert space $\mathcal{L}$ and a normal operator $N \in \mathfrak{B}(\mathcal{H}\oplus\mathcal{L})$ of the form
			\[
			N = \begin{bmatrix} T & * \\ 0 & * \end{bmatrix};
			\]
			\item hyponormal operators: $TT^* \leq T^*T$;
			\item $p$-hyponormal operators: $(TT^*)^p \leq (T^*T)^p$ for some $0<p\leq 1$;
			\item class $A$ operators: $T^*T \leq \left({T^*}^2 T^2\right)^{1/2}$;
			\item paranormal operators: $\|Tx\|^2 \leq \|T^2x\|\norm{x}$ for all $x \in \mathcal{H}$;
			\item normaloid operators: $r(T) = \|T\|$, where $r(T)$ denotes the spectral radius of $T$.
		\end{itemize}
		
		It is well known that the following proper inclusions hold:
		\begin{equation}\label{eq:inclusion_chain}
			\begin{split}
				\text{normal}
				&\subset \text{quasinormal}
				\subset \text{subnormal}
				\subset \text{hyponormal} \\
				&\subset \text{$p$-hyponormal}
				\subset \text{class $A$}
				\subset \text{paranormal} \\
				&\subset \text{normaloid}.
			\end{split}
		\end{equation}
		
		The classes of subnormal and hyponormal operators were introduced by Halmos in \cite{Halmos50}, while the study of quasinormal operators was initiated by Brown in \cite{Brown53}. The class of $p$-hyponormal operators was defined as an extension of hyponormal operators in \cite{Xia83} and has since been investigated extensively (see, e.g., \cite{Aluthge90, Aluthge96}). The concepts of paranormal operators and operators of class $A$ were introduced by Istr\v a\cb{t}escu in \cite{Istratescu67} and by Furuta, Ito, and Yamazaki in \cite{FurutaItoYamazaki98}, respectively. For further details on these operator classes, we refer the reader to \cite{Furuta01, Conway91, Kubrusly03}.
		
		Several generalizations of paranormal operators have emerged in the literature.
		\begin{definition}
			Let $k\in\mathbb{N}\cup\{0\}$. An operator \( T \in \mathfrak{B}(\mathcal{H}) \) is called 
			\( k \)-paranormal if
			\begin{equation*} 
				\norm{Tx}^{k+1}\leq\norm{T^{k+1}x}\norm{x}^k,\quad x\in\mathcal{H};
			\end{equation*}
		\end{definition}
		
		\begin{definition}
			Let $k>0$. An operator \( T \in \mathfrak{B}(\mathcal{H}) \) is called absolute-$k$-paranormal if
			\begin{equation*} 
				\norm{Tx}^{k+1}\leq\norm{|T|^kTx}\norm{x}^k,\quad x\in\mathcal{H}.
			\end{equation*}
		\end{definition}

		We note that
		\begin{align}\label{eq:1_paranormal}
			\text{paranormal}
			&\equiv \text{\( 1 \)-paranormal}
			\equiv \text{absolute-\( 1 \)-paranormal}.
		\end{align}
		The notion of \( k \)-paranormal operators was first mentioned in \cite{IstratescuIstratescu67} and later studied in \cite{DuggalKubrusly11, KubruslyDuggal10}. The class of absolute-\( k \)-paranormal operators was introduced in \cite{FurutaItoYamazaki98}.  Notably, both classes are contained within the class of normaloid operators.
		
		\medskip 
		
		There exist other natural generalizations of normality that do not necessarily fit into the standard inclusion chain \eqref{eq:inclusion_chain}. Recall that an operator $T \in \mathfrak{B}(\mathcal{H})$ is called \emph{binormal} if $T^*T$ commutes with $TT^*$. This class was introduced by Campbell in \cite{Campbell72} and further studied extensively in \cite{Campbell75, Campbell80}. The example of the $2 \times 2$ nilpotent matrix $T = \begin{bmatrix} 0 & 1 \\ 0 & 0 \end{bmatrix}$ demonstrates that a binormal operator need not be normaloid.
		In fact, binormal operators are not even hyponormal, nor are they implied by hyponormality.
        
		Another significant class is that of \emph{posinormal} operators, introduced in \cite{Rhaly94}. An operator $T \in \mathfrak{B}(\mathcal{H})$ is said to be posinormal if there exists a positive operator $P \in \mathfrak{B}(\mathcal{H})$ such that $TT^* = T^*PT$. While every hyponormal operator is posinormal \cite[Corollary 2.1]{Rhaly94}, posinormality does not, in general, imply the normaloid property. Indeed, any invertible non-normaloid operator serves as a counterexample (see \cite[Theorem 3.1]{Rhaly94}). For further details on posinormal operators, we refer the reader to \cite{KubruslyDuggal07, KubruslyVieieraZanni16}.

		\medskip 
		
		While many structural properties of normal operators are lost in the broad class of normaloid operators, subclasses within the paranormal family—such as $k$-paranormal and absolute-$k$-paranormal operators—retain several essential normal-like characteristics. Consequently, membership in these families often serves as a natural a priori assumption in various theorems investigating conditions that force an operator toward normality.
		
		A prominent example of this role is found in the study of the \emph{$n$-th root problem}: if $T$ belongs to a certain class of operators and $T^n$ is normal for some $n \in \mathbb{N}$, does it follow that $T$ is normal? This line of inquiry has a rich history. Stampfli \cite{Stampfli62} first established an affirmative answer for the class of hyponormal operators, a result later extended to the broader class of paranormal operators by Ando \cite{Ando72}. More recently, the problem was further generalized to include both $k$-paranormal and absolute-$k$-paranormal operators by Stankovi\'c and Kubrusly \cite{StankovicKubrusly25}. For a more detailed exposure of the $n$-root problem, see \cite[Section 2]{StankovicKubrusly25} and the references cited therein.
		
		\medskip 
		
		The primary objective of the present paper is to investigate whether such normality-inducing properties, along with the structural links arising from the polar decomposition, extend to the broadest family of paranormal-like normaloid operators established to date: the \emph{absolute-$(p,r)$-paranormal operators}. More precisely, 

        \begin{align*} 
			\text{normal}&\subset \cdots\subset 	\text{paranormal} \\&\subset \text{absolute-\( k \)-paranormal} \\&\subset \text{absolute-\( (p,r) \)-paranormal}\\&\subset \text{normaloid}.
		\end{align*}

		\bigskip 
		\section{Preliminaries}\label{sec:prelim}
		\bigskip

		In this section, we recall the definition and fundamental properties of absolute-$(p,r)$-paranormal operators, as well as some essential results from operator theory that will be utilized in the subsequent proofs.
		
		\begin{definition}\cite{YamazakiYanagida00}\label{def:pr_paranormal}
			For positive real numbers $p > 0$ and $r > 0$, an operator $T\in\mathfrak{B}(\mathcal{H})$ is said to be \emph{absolute-$(p,r)$-paranormal} if
			\begin{equation}\label{eq:definition_star}
				\| |T|^p |T^*|^r x \|^r \geq \| |T^*|^r x \|^{p+r}
			\end{equation}
			for every unit vector $x \in \mathcal{H}$, 
			or, equivalently,
			\begin{equation}\label{eq:definition_star_all}
				\| |T|^p |T^*|^r x \|^r\|x\|^p \geq \| |T^*|^r x \|^{p+r}
			\end{equation}
			for all $x \in \mathcal{H}$.
		\end{definition}
		
		Using the polar decomposition $T = U|T|$, the defining condition can be reformulated as follows.
		
		\begin{proposition}\cite[Proposition 1]{YamazakiYanagida00}\label{prop:def_U}
			Let $T = U |T|$ be the polar decomposition of $T\in\mathfrak{B}(\mathcal{H})$, and let $p, r > 0$. Operator $T$ is absolute-$(p,r)$-paranormal if and only if
			\begin{equation}\label{eq:definition_U}
				\| |T|^p U |T|^r x \|^r \geq \| |T|^r x \|^{p+r}
			\end{equation}
			for every unit vector $x \in \mathcal{H}$.
		\end{proposition} 
		
		Clearly, for each $k>0$,
		\begin{align*}
			\text{absolute-\( k \)-paranormal}
			\;\equiv\;
			\text{absolute-\( (k,1) \)-paranormal}.
		\end{align*}
		Consequently, by \eqref{eq:1_paranormal},
		\begin{align*}
			\text{paranormal}
			\;\equiv\;
			\text{absolute-\( (1,1) \)-paranormal}.
		\end{align*}
		
		A useful characterization of absolute-$(p,r)$-paranormal operators in terms of operator inequalities is given by the following result.
		
		\begin{proposition}\cite{YamazakiYanagida00}\label{prop:lambda_char_pq}
			Let $p, r > 0$. An operator $T\in\mathfrak{B}(\mathcal{H})$ is absolute-$(p,r)$-paranormal if and only if 
			\begin{equation}\label{eq:lambda_char_pq}
				r \, |T^*|^r |T|^{2p} |T^*|^r
				- (p+r)\lambda^p |T^*|^{2r}
				+ p \lambda^{p+r} I
				\geq 0,  \quad \text{for all } \lambda > 0.
			\end{equation} 
		\end{proposition}
		
		The next two theorems show that the class of absolute-$(p,r)$-paranormal operators consists of normaloid operators and that this class is monotone with respect to the parameters $p$ and $r$.
		
		\begin{theorem}\cite[Theorem 8]{YamazakiYanagida00}\label{thm:pr_normaloid}
			Let $p,r>0$. Every absolute-$(p,r)$-paranormal operator is normaloid.
		\end{theorem}
		
		\begin{theorem}\cite[Theorem 7]{YamazakiYanagida00}\label{thm:monotonicity}
			Let $T\in\mathfrak{B}(\mathcal{H})$ be absolute-$(p_0, r_0)$-paranormal for some $p_0, r_0 > 0$. Then $T$ is absolute-$(p, r)$-paranormal for all $p \geq p_0$ and $r \geq r_0$. 
		\end{theorem}
		
		\medskip 
		
		We conclude this section by recalling fundamental results from operator theory that provide the necessary technical framework for the forthcoming arguments.
		
		\begin{theorem}\cite[p. 58]{Furuta01}\label{thm:polar_q}
			Let $T=U|T|$ be the polar decomposition of an operator $T\in\mathfrak{B}(\mathcal{H})$. Then:
			\begin{enumerate}[label=\textit{(\roman*)}]
				\item $\mathcal{N}(|T|)=\mathcal{N}(T)$;
				\item $|T^*|^q=U|T|^qU^*$ for any $q>0$.
			\end{enumerate}
		\end{theorem}
		
		\begin{theorem}\label{thm:star_polar}\cite[ p. 59]{Furuta01} 
			Let $T = U\vert T \vert$ be the polar decomposition of $T\in\mathfrak{B}(\mathcal{H})$. Then $T^{*} = U^{*} \vert T^{*}\vert$ is also the polar decomposition of the operator $T^{*}.$
		\end{theorem}

        \begin{theorem}\label{thm:T_imply_U}\cite[ p. 75]{Furuta01} 
Let $T = U|T|$ be the polar decomposition of an operator $T$. Then
\begin{enumerate}[label=\textit{(\roman*)}]
    \item If $T$ is binormal, then so is $U$.
    \item If $T$ is quasinormal, then so is $U$; $U = \text{isometry} \oplus 0$ on $N(T)^{\perp} \oplus N(T)$.
    \item If $T$ is normal, then so is $U$; $U = \text{unitary} \oplus 0$ on $N(T)^{\perp} \oplus N(T)$.
    \item If $T$ is self-adjoint, then so is $U$; $U = \text{symmetry} \oplus 0$ on $N(T)^{\perp} \oplus N(T)$.
    \item If $T$ is positive, then so is $U$; $U = \text{projection}$.
\end{enumerate}
\end{theorem}

		\begin{theorem}\cite{McCarthy67} \cite[p. 123]{Furuta01}[H\"older-McCarthy inequality]\label{thm:holder_mccarthy_ineq}
			Let $A\in\mathfrak{B}(\mathcal{H})$ be a positive operator. Then for any unit vector $x \in \mathcal{H}$:
			\begin{enumerate}[label=\textit{(\roman*)}]
				\item $\langle A^\alpha x, x \rangle \geq \langle A x, x \rangle^\alpha$ for $\alpha > 1$;
				\item $\langle A^\alpha x, x \rangle \leq \langle A x, x \rangle^\alpha$ for $\alpha \in [0,1]$.
			\end{enumerate}
		\end{theorem}
		
		\medskip 
		
		The rest of the paper is organized as follows.  
		In Section~\ref{sec:self_polar}, we investigate the relationship between the self-adjointness of an operator $T$ and its polar factor $U$ within the class of absolute-$(p,r)$-paranormal operators. 
		In Section~\ref{sec:powers}, we address, among other topics, the $n$-th root problem by extending Ando's Theorem to absolute-$(p,r)$-paranormal operators, and we also present several results concerning compactness.  
		Finally, Section~\ref{sec:partial_iso} is devoted to new characterizations of quasinormal partial isometries.

		\bigskip 
		\section{Self-adjointness criteria via polar decomposition}\label{sec:self_polar}
		\bigskip 
		
		In recent papers by the first author \cite{Stankovic24, Stankovic25}, the following result was established.
		
		\begin{theorem}\cite[Corollary~2.1]{Stankovic24}\cite[Theorem~1.7]{Stankovic25}\label{thm:u_self_adjoint_p_hypo}
			Let $T = U|T|$ be the polar decomposition of an operator $T \in \mathfrak{B}(\mathcal{H})$, and let $0 < p \leq 1$. The following conditions are equivalent:
			\begin{enumerate}[label=\textit{(\roman*)}]
				\item $T$ is self-adjoint;
				\item $T$ is $p$-hyponormal and $U$ is self-adjoint.
			\end{enumerate}
		\end{theorem}
		
		The proof in \cite[Theorem~1.7]{Stankovic25} relies on intrinsic operator inequalities associated with $p$-hyponormal operators, whereas the approach in \cite[Corollary~2.1]{Stankovic24} is based on the Aluthge transform. Moreover, \cite[Theorem~2.6]{Stankovic24} shows that conditions \textit{(i)} and \textit{(ii)} of Theorem~\ref{thm:u_self_adjoint_p_hypo} are also equivalent to the following condition:
		\begin{itemize}
			\item[$(iii)'$] $T$ is $p$-hyponormal and
			\[
			U^*|T|U \leq \operatorname{Re}(U)\,|T|\,\operatorname{Re}(U).
			\]
		\end{itemize}
		
		The next result extends Theorem~\ref{thm:u_self_adjoint_p_hypo} to a broader class of absolute-$(p,r)$-paranormal operators.
		
		\begin{theorem}\label{thm:u_self_adjoint_pq}
			Let $T = U|T|$ be the polar decomposition of an operator $T \in \mathfrak{B}(\mathcal{H})$, and let $p,r>0$. The following conditions are equivalent:
			\begin{enumerate}[label=\textit{(\roman*)}]
				\item $T$ is self-adjoint;
				\item $T$ is absolute-$(p,r)$-paranormal and $U$ is self-adjoint.
			\end{enumerate}
		\end{theorem}
		
		\begin{proof}
			The implication $(i)\Rightarrow(ii)$ follows directly from Theorem \ref{thm:T_imply_U}, so it suffices to prove $(ii)\Rightarrow(i)$.
			
			Since $T$ is absolute-$(p,r)$-paranormal, \eqref{eq:definition_star_all} implies that for every $y \in \mathcal{H}$,
			\begin{equation}\label{eq:paranormal_y}
				\| |T^*|^r y \|^{p+r} \leq \| |T|^p |T^*|^r y \|^r \, \|y\|^p.
			\end{equation}
			Let $x \in \mathcal{H}$ be arbitrary. Substituting $y = Ux$ into \eqref{eq:paranormal_y}, we obtain
			\[
			\| |T^*|^r Ux \|^{p+r} \leq \| |T|^p |T^*|^r Ux \|^r \, \|Ux\|^p.
			\]
			Since $T^* = U|T^*|$ is the polar decomposition of $T^*$ by Theorem \ref{thm:star_polar}, and $U=U^*$, it follows that $U^2 = UU^*$ is the orthogonal projection onto $\overline{\mathcal{R}(T)} = \overline{\mathcal{R}(|T^*|^r)}$. 
			Hence $U^2 |T^*|^r = |T^*|^r$, and therefore, using Theorem \ref{thm:polar_q},
			\begin{align*}
				\| |T|^r x \|^{p+r}
				&= \| U|T^*|^r Ux \|^{p+r}
				= \| |T^*|^r Ux \|^{p+r} \\
				&\leq \| U|T|^p U^2 |T^*|^r Ux \|^r \, \|Ux\|^p \\
				&\leq \| (U|T|^p U)(U|T^*|^r U)x \|^r \, \|U\|^p \, \|x\|^p \\
				&\leq \| |T^*|^p |T|^r x \|^r \, \|x\|^p.
			\end{align*}
			This shows that $T^*$ is also $(p,r)$-paranormal. By \cite[Theorem 2]{YamazakiYanagida03}, it follows that $T$ is normal, and hence $|T| = |T^*|$. Consequently,
			\[
			T^* = U|T| = T,
			\]
			which proves that $T$ is self-adjoint.
		\end{proof}
		
		It remains unclear whether an analogue of condition $(iii)'$ is equivalent to conditions $(i)$ and $(ii)$ in Theorem~\ref{thm:u_self_adjoint_pq}.
		
		\begin{example}\label{ex:normaloid}
			The conclusion of Theorem~\ref{thm:u_self_adjoint_pq} fails if $(p,r)$-paranormality is replaced by normaloidness. Indeed, consider
			\[
			T = \begin{bmatrix}
				2 & 0 & 0 \\
				0 & 0 & 2 \\
				0 & 1 & 0
			\end{bmatrix}.
			\]
			Then
			\[
			T^*T
			= \begin{bmatrix}
				4 & 0 & 0 \\
				0 & 1 & 0 \\
				0 & 0 & 4
			\end{bmatrix},
			\]
			and hence $\|T\| = \sqrt{\|T^*T\|} = 2$. Since the eigenvalues of $T$ are $\{2,\pm\sqrt{2}\}$, we have $r(T) = 2$, so $T$ is normaloid.
			
			On the other hand,
			\[
			U = T|T|^{-1}
			= \begin{bmatrix}
				1 & 0 & 0 \\
				0 & 0 & 1 \\
				0 & 1 & 0
			\end{bmatrix},
			\]
			which is self-adjoint, while $T$ is not.
		\end{example}
		
		The choice of a $3 \times 3$ matrix in the previous example is essential, as the same conclusion fails in lower dimensions. Indeed, in the two-dimensional case, the implication holds as a direct consequence of the following result, which identifies normality and normaloidness for $2 \times 2$ matrices.
		
		\begin{theorem}
			Let $T$ be a $2\times 2$ matrix. Then the following conditions are equivalent:
			\begin{enumerate}[label=\textit{(\roman*)}]
				\item $T$ is normal;
				\item $T$ is normaloid.
			\end{enumerate}
		\end{theorem}
		
		\begin{proof}
			We only prove $(ii)\Rightarrow(i)$.
			Assume that $T$ is normaloid. By the Schur decomposition, $T$ is unitarily similar to a matrix
			\[
			\widehat{T}=
			\begin{bmatrix}
				\lambda_1 & \alpha \\
				0 & \lambda_2
			\end{bmatrix},
			\]
			where $\lambda_1,\lambda_2 \in \mathbb{C}$ are the eigenvalues of $T$ and $\alpha \in \mathbb{C}$. Since normaloidness is preserved under unitary equivalence, $\widehat{T}$ is also normaloid. Without loss of generality, assume $|\lambda_1| \geq |\lambda_2|$, so that $\|\widehat{T}\| = r(\widehat{T}) = |\lambda_1|$.
			
			A direct computation yields
			\[
			\widehat{T}^*\widehat{T}
			=
			\begin{bmatrix}
				\overline{\lambda_1} & 0 \\
				\overline{\alpha} & \overline{\lambda_2}
			\end{bmatrix}
			\begin{bmatrix}
				\lambda_1 & \alpha \\
				0 & \lambda_2
			\end{bmatrix}
			=
			\begin{bmatrix}
				|\lambda_1|^2 & \overline{\lambda_1}\alpha \\
				\overline{\alpha}\lambda_1 & |\alpha|^2+|\lambda_2|^2
			\end{bmatrix}.
			\]
			The characteristic polynomial of $\widehat{T}^*\widehat{T}$ is given by
			\[
			\det(\widehat{T}^*\widehat{T}-\lambda I)
			=
			\lambda^2-\operatorname{tr}(\widehat{T}^*\widehat{T})\lambda
			+\det(\widehat{T}^*\widehat{T}),
			\]
			and therefore
			\begin{align*}
				\|\widehat{T}\|^2
				&=
				\dfrac{\operatorname{tr}(\widehat{T}^*\widehat{T})
					+\sqrt{[\operatorname{tr}(\widehat{T}^*\widehat{T})]^2
						-4\det(\widehat{T}^*\widehat{T})}}{2} \\
				&=
				\dfrac{|\lambda_1|^2+|\lambda_2|^2+|\alpha|^2
					+\sqrt{\left(|\lambda_1|^2+|\lambda_2|^2+|\alpha|^2\right)^2
						-4|\lambda_1|^2|\lambda_2|^2}}{2}.
			\end{align*}
			Since $\|\widehat{T}\|=|\lambda_1|$, we obtain
			\begin{align*}
				|\lambda_1|^2
				&=
				\dfrac{|\lambda_1|^2+|\lambda_2|^2+|\alpha|^2
					+\sqrt{\left(|\lambda_1|^2+|\lambda_2|^2+|\alpha|^2\right)^2
						-4|\lambda_1|^2|\lambda_2|^2}}{2} \\
				&\geq
				\dfrac{|\lambda_1|^2+|\lambda_2|^2
					+\sqrt{\left(|\lambda_1|^2+|\lambda_2|^2\right)^2
						-4|\lambda_1|^2|\lambda_2|^2}}{2} \\
				&=
				\dfrac{|\lambda_1|^2+|\lambda_2|^2
					+\sqrt{\left(|\lambda_1|^2-|\lambda_2|^2\right)^2}}{2} \\
				&=
				\dfrac{|\lambda_1|^2+|\lambda_2|^2+|\lambda_1|^2-|\lambda_2|^2}{2}
				=|\lambda_1|^2.
			\end{align*}
			Hence,
			\begin{align*}
				&\quad
				|\lambda_1|^2+|\lambda_2|^2+|\alpha|^2
				+\sqrt{\left(|\lambda_1|^2+|\lambda_2|^2+|\alpha|^2\right)^2
					-4|\lambda_1|^2|\lambda_2|^2} \\
				&=
				|\lambda_1|^2+|\lambda_2|^2
				+\sqrt{\left(|\lambda_1|^2+|\lambda_2|^2\right)^2
					-4|\lambda_1|^2|\lambda_2|^2}.
			\end{align*}
			Since
			\[
			|\lambda_1|^2+|\lambda_2|^2+|\alpha|^2 \geq |\lambda_1|^2+|\lambda_2|^2
			\]
			and
			\[
			\sqrt{\left(|\lambda_1|^2+|\lambda_2|^2+|\alpha|^2\right)^2
				-4|\lambda_1|^2|\lambda_2|^2}
			\geq
			\sqrt{\left(|\lambda_1|^2+|\lambda_2|^2\right)^2
				-4|\lambda_1|^2|\lambda_2|^2},
			\]
			both inequalities must be equalities. Consequently,
			\[
			|\lambda_1|^2+|\lambda_2|^2+|\alpha|^2
			=
			|\lambda_1|^2+|\lambda_2|^2,
			\]
			which forces $\alpha=0$. Thus, $\widehat{T}$ is normal, and since normality is preserved
			under unitary similarity, $T$ is normal as well. This completes the proof.
		\end{proof}

		\bigskip 
		
		\section{Powers of absolute-$(p,r)$-paranormal operators}\label{sec:powers}
		\bigskip

		We begin by recalling some standard nomenclature: an operator $A \in \mathfrak{B}(\mathcal{H})$ is referred to as a \textit{scalar operator} if it is a complex multiple of the identity, i.e., $A = \lambda I$ for some $\lambda \in \mathbb{C}$.
		
		\begin{lemma}\label{lem:normaloid_scalar}
			Let $T \in \mathfrak{B}(\mathcal{H})$ be a normaloid operator. If $T^n$ is a scalar operator for some $n\in\mathbb{N}$, then $T$ is normal. Moreover, if $T$ is non-zero, it is a non-zero multiple of a unitary operator.
		\end{lemma}
		
		\begin{proof}
			Since $T^n$ is a scalar operator, we may write $T^n = \lambda I$ for some $\lambda \in \mathbb{C}$. If $\lambda = 0$, then $T$ is nilpotent. Since the only nilpotent normaloid operator is the null operator, the result follows trivially for $T = 0$.
			
			Suppose now that $\lambda \neq 0$. We define the normalized operator $S = \lambda^{-1/n}T$, which is clearly normaloid and satisfies $S^n = I$. The power identity $S^n = I$ implies that the spectrum $\sigma(S)$ is contained within the unit circle in the complex plane.
			
			Consequently, $S$ is invertible and its spectral radius satisfies $r(S) = 1$. The normaloid property of $S$ then ensures that $\|S\| = r(S) = 1$. Furthermore, the condition $S^n = I$ combined with $\|S\| = 1$ implies that $\|S^k\| = 1$ for every $k \in \{1, \dots, n\}$, since
			\[
			1 = \|S^n\| \leq \|S^{n-k}\| \|S^k\| \leq \|S\|^{n-k} \|S^k\| = \|S^k\| \leq \|S\|^k = 1.
			\]
			The relation $S^n = I$ also implies $S^{-1} = S^{n-1}$. Given that $\|S^{n-1}\| = 1$, it follows that $\|S^{-1}\| = 1$. Thus, $S$ is an invertible contraction whose inverse is also a contraction. On a Hilbert space, such an operator is necessarily unitary (see, e.g., \cite[Problem 2.6]{Kubrusly03}). Therefore, $T = \lambda^{1/n}S$ is a scalar multiple of a unitary operator and is, consequently, normal.
		\end{proof}

		\begin{corollary}\label{cor:root_of_scalar}
			Let $T\in\mathfrak{B}(\mathcal{H})$ be an absolute-$(p,r)$-paranormal operator for some $p,r>0$. If $T^n$ is a scalar operator for some $n\in\mathbb{N}$, then $T$ is normal.
		\end{corollary}
		
		\begin{proof}
			The proof follows directly from Lemma \ref{lem:normaloid_scalar} and Theorem \ref{thm:pr_normaloid}.
		\end{proof}

		The following theorem extends \cite[Theorem 3.1]{StankovicKubrusly25} and employs an analogous proof strategy.
		
		\begin{theorem}\label{thm:pq_paranormal_root}
			Let $T\in\mathfrak{B}(\mathcal{H})$ be an absolute-$(p,r)$-paranormal operator for some $p,r>0$. If there exists $n\in\mathbb{N}$ such that $T^n$ is normal, then $T$ is normal.
		\end{theorem}

		\begin{proof}
			First assume that the Hilbert space $\mathcal{H}$ is separable.
			Let $\mathcal{A}$ denote the abelian von Neumann algebra generated by $T^n$.   By the Spectral Theorem, the Hilbert space $\mathcal{H}$ can be identified with a direct integral
			\[
			\mathcal{H} = \int_{\sigma(T^n)}^\oplus \mathcal{H}_\lambda \, d\mu(\lambda),
			\]
			where each operator in $\mathcal{A}$ acts as a multiplication operator on this decomposition. Specially, 
			\begin{equation*}
				T^n = \int_{\sigma(T^n)}^\oplus \lambda \, d\mu(\lambda).
			\end{equation*}
			Furthermore, the commutant $\mathcal{A}'$ of $\mathcal{A}$ is decomposable relative to this representation. Since $T \in \mathcal{A}'$, it follows that $T$ can be expressed as
			\[
			T = \int_{\sigma(T^n)}^\oplus T_\lambda \, d\mu(\lambda),
			\]
			where $T_\lambda$ is an operator on $\mathcal{H}_\lambda$ for all $\lambda\in\sigma(T^n)$. 
			Therefore, we have that 
			\begin{equation*} 
				T_\lambda^n = \lambda I_\lambda,
			\end{equation*}
			for $\mu$-almost all $\lambda\in\sigma(T^n)$, 
			where $I_\lambda$ denotes the identity operator on Hilbert space $\mathcal{H}_\lambda$, $\lambda\in\sigma(T^n)$. 
			
			Now let $\zeta>0$ be arbitrary. Since $T$ is absolute-$(p,r)$-paranormal, it satisfies \eqref{eq:lambda_char_pq}, and thus,
			\begin{equation*}
				\int_{\sigma(T^n)}^\oplus\left[r \, |T_\lambda^*|^r |T_\lambda|^{2p} |T_\lambda^*|^r - (p+r)\zeta^p |T_\lambda^*|^{2r} + p \zeta^{p+r} I_\lambda\right]\,d\mu(\lambda)\geq 0.
			\end{equation*}
			By \cite[Proposition 2.6.1]{Nielsen80}, we have that
			\begin{equation*} 
				r \, |T_\lambda^*|^r |T_\lambda|^{2p} |T_\lambda^*|^r - (p+r)\zeta^p |T_\lambda^*|^{2r} + p \zeta^{p+r} I_\lambda\geq 0,\quad \zeta>0,
			\end{equation*}
			for $\mu$-almost all $\lambda\in\sigma(T^n)$. Proposition \ref{prop:lambda_char_pq} implies that $T_\lambda$ is absolute-$(p,r)$-paranormal for $\mu$-almost all $\lambda\in\sigma(T^n)$. By Corollary \ref{cor:root_of_scalar},   $T_\lambda$ is normal for $\mu$-almost all $\lambda\in\sigma(T^n)$, and consequently, $T$ is also normal.
			
			The general case reduces to the separable one using the same technique employed in \cite[Theorem 3.1]{StankovicKubrusly25}. (See also the proof of Theorem \ref{thm:pq_hypo} below.)
		\end{proof}

		The following theorem represents a generalization of \cite[Theorem 4]{Campbell75}.
		\begin{theorem}\label{thm:pq_hypo}
			Let $T\in\mathfrak{B}(\mathcal{H})$ be an absolute-$(p,r)$-paranormal for some $p,r>0$. If $T$ is binormal, then $T$ is hyponormal.
		\end{theorem}
		
		\begin{proof}
			Again, we first assume that the Hilbert space $\mathcal{H}$ is separable. Since $T$ is binormal, the positive operators $|T|^{2p}$ and $|T^*|^{r}$ commute. Consequently,
			inequality \eqref{eq:lambda_char_pq} reduces to
			\begin{equation}\label{eq:pr_binormal}
				r(T^*T)^p (TT^*)^r - (p+r)\zeta^p (TT^*)^r + p \zeta^{p+r} I \geq 0
				\quad \text{for all } \zeta > 0.
			\end{equation}
			
			Let $\mathcal{A}$ denote the abelian von Neumann algebra generated by the commuting positive
			operators $T^*T$ and $TT^*$. By the Spectral Theorem, the Hilbert space $\mathcal{H}$ can be
			identified with a direct integral
			\[
			\mathcal{H} = \int_\Lambda^\oplus \mathcal{H}_\lambda \, d\mu(\lambda),
			\]
			with respect to which each operator in $\mathcal{A}$ acts as a multiplication operator.
			In particular,
			\[
			T^*T = \int_\Lambda^\oplus f(\lambda) I_\lambda \, d\mu(\lambda),
			\qquad
			TT^* = \int_\Lambda^\oplus g(\lambda) I_\lambda \, d\mu(\lambda),
			\]
			for some non-negative functions $f,g \in L^\infty(X)$, where $I_\lambda$ denotes the identity
			operator on the fiber $\mathcal{H}_\lambda$, $\lambda\in \Lambda$.
			
			With respect to this representation, inequality \eqref{eq:pr_binormal} becomes
			\[
			\int_\Lambda^\oplus
			\left[
			r (f(\lambda))^p (g(\lambda))^r
			- (p+r)\zeta^p (g(\lambda))^r
			+ p \zeta^{p+r}
			\right] I_\lambda \, d\mu(\lambda)
			\geq 0
			\quad \text{for all } \zeta > 0.
			\]
			Hence, for $\mu$-almost every $\lambda\in \Lambda$, by \cite[Proposition 2.6.1]{Nielsen80}, we have
			\begin{equation}\label{eq:pr_lambda}
				r (f(\lambda))^p (g(\lambda))^r
				- (p+r)\zeta^p (g(\lambda))^r
				+ p \zeta^{p+r}
				\geq 0
				\quad \text{for all } \zeta > 0.
			\end{equation}
			
			Fix such a $\lambda\in \Lambda$ and set $x=f(\lambda)$ and $y=g(\lambda)$. Then \eqref{eq:pr_lambda} can be written as
			\begin{equation}\label{eq:xy_lambda}
				r x^p y^r - (p+r)\zeta^p y^r + p \zeta^{p+r} \geq 0
				\quad \text{for all } \zeta > 0.
			\end{equation}
			We claim that $x\geq y$. If $y=0$,
			the claim is immediate. Thus, assume that $y>0$. 
			Taking $\zeta=y$ in \eqref{eq:xy_lambda}, we obtain
			\[
			r x^p y^p - (p+r)y^{p+r} + p y^{p+r} \geq 0,
			\]
			which simplifies to
			\[
			r x^p y^p \geq r y^{p+r}.
			\]
			Since $r,y>0$, it follows immediately that $x\geq y$. Therefore,
			$f(\lambda)\geq g(\lambda)$ for $\mu$-almost every $\lambda\in \Lambda$.
			
			Finally, this yields
			\[
			T^*T
			=
			\int_\Lambda^\oplus f(\lambda) I_\lambda \, d\mu(\lambda)
			\geq
			\int_\Lambda^\oplus g(\lambda) I_\lambda \, d\mu(\lambda)
			=
			TT^*,
			\]
			showing that $T$ is hyponormal.
			
			We now turn to the general case where $\mathcal{H}$ is not necessarily separable. Let $y\in\mathcal{H}$ be arbitrary. Define $\mathcal{H}_y\subseteq\mathcal{H}$ as
			\begin{equation*}
				\mathcal{H}_y:=\bigvee\left\{{T^*}^{i_k}T^{j_k}\cdots{T^*}^{i_1}T^{j_1}y:\, (i_1,\ldots,i_k),(j_1,\ldots,j_k)\in\mathbb{N}_0^k,\, k\in\mathbb{N}\right\},
			\end{equation*}
			where $\bigvee \mathcal{S}$ denotes the closure of the linear span of vectors in a set $\mathcal{S}\subseteq\mathcal{H}$. Then, it is easy to see that $\mathcal{H}_y$ is a separable Hilbert space and it reduces $T$. Let $T_y:=T\restriction_{\mathcal{H}_y}$. Note that $T_y$ satisfies \eqref{eq:definition_star} for all $x\in\mathcal{H}_y$. Thus, $T_y$ is absolute-$(p,r)$-paranormal, and it is evident that it is also binormal. Consequently, $T_y$ is hyponormal. Therefore,
			
			\begin{equation*}
				\skal{TT^*y}{y}=\skal{T_yT^*_yy}{y}\leq \skal{T^*_yT_yy}{y}=\skal{T^*Ty}{y}.
			\end{equation*} 
			Since $y\in\mathcal{H}$ was arbitrary, we conclude that $TT^*\leq T^*T$. In other words, $T$ is hyponormal.
		\end{proof}

		\begin{example}
			Let $T$ be the normaloid matrix given in Example~\ref{ex:normaloid}. As already computed,
			\[
			T^*T
			= \begin{bmatrix}
				4 & 0 & 0 \\
				0 & 1 & 0 \\
				0 & 0 & 4
			\end{bmatrix},
			\]
			while
			\[
			TT^*
			= \begin{bmatrix}
				4 & 0 & 0 \\
				0 & 4 & 0 \\
				0 & 0 & 1
			\end{bmatrix}.
			\]
			Since both $T^*T$ and $TT^*$ are diagonal matrices, they commute, and hence $T$ is binormal. On the other hand, $T^*T \neq TT^*$, so $T$ is not normal.
			
			Recall that in finite-dimensional Hilbert spaces normality is equivalent to hyponormality. Therefore, this example shows that Theorem~\ref{thm:pq_hypo} cannot be extended to the class of normaloid operators.
		\end{example}
		
		\begin{theorem}\label{thm:lambda_squared}
			Let $T\in\mathfrak{B}(\mathcal{H})$ satisfy
			\begin{equation}\label{eq:leq_lambda}
				TT^*\leq \lambda\, T^*T
			\end{equation}
			for some $\lambda>0$. If $T$ is binormal, then
			\[
			T^n{T^*}^n\leq \lambda^{n^2}{T^*}^nT^n
			\]
			for every $n\in\mathbb{N}$.
		\end{theorem}
		
		\begin{proof}
			Fix $n\in\mathbb{N}$. Since $T$ is binormal, the operators $TT^*$ and $T^*T$ commute. Hence, by functional calculus, \eqref{eq:leq_lambda} implies
			\begin{equation}\label{eq:leq_lambda_k}
				(TT^*)^k\leq \lambda^k (T^*T)^k
			\end{equation}
			for all $k\in\mathbb{N}$.
			
			Furthermore,  successive  application of \eqref{eq:leq_lambda_k} yields
			\begin{align*}
				(T^*T)^n
				&= T^*(TT^*)^{\,n-1}T \\
				&\le \lambda^{n-1} T^*(T^*T)^{\,n-1}T \\
				&= \lambda^{n-1} {T^*}^2 (TT^*)^{\,n-2} T^2 \\
				&\le \lambda^{(n-1)+(n-2)} {T^*}^2 (T^*T)^{\,n-2} T^2 \\
				&\;\;\vdots \\
				&\le \lambda^{(n-1)+(n-2)+\cdots+1} {T^*}^n T^n \\
				&= \lambda^{\frac{n(n-1)}{2}} {T^*}^n T^n .
			\end{align*}
			In analogous manner, we can derive
			\[
			(TT^*)^n \ge \lambda^{-\frac{n(n-1)}{2}} T^n {T^*}^n .
			\]
			
			Combining these inequalities with \eqref{eq:leq_lambda_k} for $k=n$, we conclude
			\begin{align*}
				T^n{T^*}^n
				&\le \lambda^{\frac{n(n-1)}{2}} (TT^*)^n \\
				&\le \lambda^{\frac{n(n-1)}{2}} \lambda^n (T^*T)^n \\
				&\le \lambda^{n^2} {T^*}^n T^n ,
			\end{align*}
			completing the proof.
		\end{proof}

		The preceding result yields two important corollaries regarding the preservation of operator properties under powers.
		\begin{corollary}\label{cor:bi_posi}
			Let $T\in\mathfrak{B}(\mathcal{H})$ be a posinormal operator. If $T$ is binormal, then $T^n$ is posinormal for all $n\in\mathbb{N}$.
		\end{corollary}
		\begin{proof}
			The assertion follows directly from Theorem~\ref{thm:lambda_squared} in conjunction with \cite[Theorem 2.1]{Rhaly94}, which says that \eqref{eq:leq_lambda} is an equivalent definition of posinormal operators.
		\end{proof}
		
		\begin{corollary}\label{cor:bi_hypo}
			Let $T\in\mathfrak{B}(\mathcal{H})$ be a hyponormal operator. If $T$ is binormal, then $T^n$ is hyponormal for all $n\in\mathbb{N}$.
		\end{corollary}
		
		It is worth noting that Corollary~\ref{cor:bi_posi} provides a generalization of \cite[Theorem 3]{KubruslyVieieraZanni16}, while   Corollary~\ref{cor:bi_hypo} recovers  \cite[Theorem 3]{Campbell75}.
		
		\begin{theorem}
			Let $p_1, r_1, p_2, r_2>0$.
			If $T\in\mathfrak{B}(\mathcal{H})$ is binormal and absolute-$(p_1,r_1)$-paranormal, and there exists $n\in\mathbb{N}$ such that ${T^*}^n$ is absolute-$(p_2,r_2)$-paranormal, then $T$ is normal. 
		\end{theorem}
		
		\begin{proof}
			By Theorem~\ref{thm:pq_hypo}, operator $T$ is hyponormal. It follows from Corollary \ref{cor:bi_hypo} that $T^k$ is hyponormal for every $k\in\mathbb{N}$. In particular, $T^n$ is hyponormal and hence paranormal.
			
			Applying \cite[Theorem~2]{YamazakiYanagida03}, we conclude that $T^n$ is normal. Consequently, Theorem~\ref{thm:pq_paranormal_root} yields that $T$ itself is normal.
		\end{proof}
		
		In order to prove our next result, we first establish the following auxiliary lemma, which is most likely well known.
		
		\begin{lemma}\label{lem:fundamendal_alpha}
			Let $T\in\mathfrak{B}(\mathcal{H})$. Then
			\[
			T|T|^\alpha = |T^*|^\alpha T
			\]
			for every $\alpha>0$.
		\end{lemma}
		
		\begin{proof}
			Let $\alpha>0$ be arbitrary. Since $T(T^*T)=(TT^*)T$, we obtain
			\[
			\begin{bmatrix}
				0 & T \\
				0 & 0
			\end{bmatrix}
			\begin{bmatrix}
				|T^*|^2 & 0 \\
				0 & |T|^2
			\end{bmatrix}
			=
			\begin{bmatrix}
				|T^*|^2 & 0 \\
				0 & |T|^2
			\end{bmatrix}
			\begin{bmatrix}
				0 & T \\
				0 & 0
			\end{bmatrix}.
			\]
			Thus, the operator
			\[
			\begin{bmatrix}
				0 & T \\
				0 & 0
			\end{bmatrix}
			\]
			commutes with the positive operator
			\[
			\begin{bmatrix}
				|T^*|^2 & 0 \\
				0 & |T|^2
			\end{bmatrix},
			\]
			and hence also with its fractional power
			\[
			\begin{bmatrix}
				|T^*|^2 & 0 \\
				0 & |T|^2
			\end{bmatrix}^{\alpha/2}.
			\]
			Consequently,
			\[
			\begin{bmatrix}
				0 & T \\
				0 & 0
			\end{bmatrix}
			\begin{bmatrix}
				|T^*|^\alpha & 0 \\
				0 & |T|^\alpha
			\end{bmatrix}
			=
			\begin{bmatrix}
				|T^*|^\alpha & 0 \\
				0 & |T|^\alpha
			\end{bmatrix}
			\begin{bmatrix}
				0 & T \\
				0 & 0
			\end{bmatrix},
			\]
			which yields $T|T|^\alpha = |T^*|^\alpha T$.
		\end{proof}
		
		\begin{theorem}\label{thm:compact_square_root}
			Let $T\in\mathfrak{B}(\mathcal{H})$ be an absolute-$(p,r)$-paranormal operator for some $p,r>0$. If $T^2$ is compact, then $T$ is compact.
		\end{theorem}
		
		\begin{proof}
			Set $s=\max\{p,r,1\}$ and let $T=U|T|$ be the polar decomposition of $T$.  By Theorem~\ref{thm:monotonicity}, the operator $T$ is absolute-$(s,s)$-paranormal.
			Therefore, by \eqref{eq:definition_U}, we have
			\begin{equation}\label{eq:power_s}
				\|\,|T|^s x\|^2 \leq \|\,|T|^s U |T|^s x\|, \qquad \|x\|=1.
			\end{equation}
			Since $s\geq 1$, Theorem~\ref{thm:holder_mccarthy_ineq} yields
			\[
			\|Tx\|^{2s}
			=\langle |T|^2 x,x\rangle^s
			\leq \langle |T|^{2s}x,x\rangle
			=\||T|^s x\|^2.
			\]
			Combining this inequality with \eqref{eq:power_s}, we obtain
			\[
			\|Tx\|^{2s} \leq \|\,|T|^s U |T|^s x\|, \qquad \|x\|=1,
			\]
			or equivalently,
			\begin{equation}\label{eq:power_s_all_x}
				\|Tx\|^{2s} \leq \|\,|T|^s U |T|^s x\|\,\|x\|^{2s-1},
			\end{equation}
			for all $x\in\mathcal{H}$. 
			
			Let $\{x_n\}_{n\in\mathbb{N}}$ be a sequence in $\mathcal{H}$ converging weakly to $0$. Without loss of generality, we may assume that $\|x_n\|\leq 1$ for all $n\in\mathbb{N}$. Since $T^2$ is compact, the operator $|T^*|^{s-1}T^2|T|^{s-1}$ is also compact. On the other hand, by Lemma~\ref{lem:fundamendal_alpha}, we have
			
			\begin{equation}\label{eq:trans_equiv}
				U|T|^s U|T|^s
				=
				T|T|^{s-1}T|T|^{s-1}
				=
				|T^*|^{s-1}T^2|T|^{s-1}.
			\end{equation}
			Therefore,
			\[
			\|\,|T|^s U|T|^s x_n\|
			=
			\|\,|T^*|^{s-1}T^2|T|^{s-1}x_n\|
			\rightarrow 0,
			\quad n\to\infty.
			\]
			Using inequality~\eqref{eq:power_s_all_x}, we obtain
			\[
			\|Tx_n\|^{2s}
			\leq
			\|\,|T|^s U|T|^s x_n\|\,\|x_n\|^{2s-1}
			\leq
			\|\,|T|^s U|T|^s x_n\|,
			\quad n\in\mathbb{N},
			\]
			and hence
			\[
			\|Tx_n\|\rightarrow 0,
			\quad n\to\infty.
			\]
			This shows that $T$ is compact.
		\end{proof}
		
		The following theorem is established by utilizing the operator transform from \cite{FujiiJungLee00}.
		
		\begin{theorem}\label{thm:compact_normal}
			Let $T\in\mathfrak{B}(\mathcal{H})$ be an absolute-$(p,r)$-paranormal operator for some $p,r>0$. If $T$ is compact, then $T$ is normal.
		\end{theorem}
		
		\begin{proof}
			Again, set $s=\max\{p,r,1\}$, and define $\overline{T}(s):=U|T|^s=T|T|^{s-1}$. Since $T$ is compact, it follows that $\overline{T}(s)$ is also compact. Moreover, since $T$ is absolute-$(s,s)$-paranormal, it is straightforward to check that $\overline{T}(s)$ is a paranormal operator. By \cite[Theorem 2]{IstratescuSaitoYoshino66}, we conclude that $\overline{T}(s)$ is normal.
			
			Now, by Theorem \ref{thm:polar_q}, we have
			\[
			{\overline{T}(s)}^*\,\overline{T}(s)
			=|T|^sU^*U|T|^s
			=|T|^{2s},
			\]
			and
			\[
			\overline{T}(s)\,{\overline{T}(s)}^*
			=U|T|^{2s}U^*
			=|T^*|^{2s}.
			\]
			Thus, $|T|^{2s}=|T^*|^{2s}$, which implies that $T$ is normal.
		\end{proof}

        \begin{remark}
            Since there exist normaloid matrices that are not normal, the condition of being absolute-$(p,r)$-paranormal in the preceding theorem cannot be relaxed to normaloidness.
        \end{remark}
		
		\begin{corollary}\label{cor:compact_root_normal}
			Let $T\in\mathfrak{B}(\mathcal{H})$ be an absolute-$(p,r)$-paranormal operator for some $p,r>0$. If $T^2$ is compact, then $T$ is normal.
		\end{corollary}
		
		\begin{proof}
			The conclusion follows immediately from Theorems \ref{thm:compact_square_root} and \ref{thm:compact_normal}.
		\end{proof}
		
		We conclude the section with the following question of interest.
		
		\begin{question}\label{quest:compact_root}
			Let $T\in\mathfrak{B}(\mathcal{H})$ be an absolute-$(p,r)$-paranormal operator for some $p,r>0$. If there exists $n\in\mathbb{N}$ such that $T^n$ is compact, does it follow that $T$ is compact?
		\end{question}
		
		Of course, an affirmative answer to Question \ref{quest:compact_root} would yield a generalization of Corollary \ref{cor:compact_root_normal} to arbitrary positive integer powers.

		\bigskip 
		\section{Characterizations of quasinormal partial isometries}\label{sec:partial_iso}
		\bigskip 
		
		It is a classical result in operator theory that for the class of partial isometries, quasinormality and hyponormality are equivalent. To see this, let $V \in \mathfrak{B}(\mathcal{H})$ be a hyponormal partial isometry. Then $VV^* \leq V^*V$. Since both $VV^*$ and $V^*V$ are orthogonal projections, this inequality is equivalent to the inclusion $\mathcal{R}(V) \subseteq \mathcal{N}(V)^\perp$. This inclusion implies that the subspace $\mathcal{N}(V)^\perp$ is invariant under $V$, which in turn ensures that $V$ commutes with the projection $V^*V$. Consequently, with respect to the orthogonal decomposition $\mathcal{H} = \mathcal{N}(V)^\perp \oplus \mathcal{N}(V)$, the operator $V$ admits the following block matrix representation:
		\begin{equation}\label{eq:iso_plus_zero}
			V = \begin{bmatrix}
				U & 0 \\
				0 & 0
			\end{bmatrix},
		\end{equation}
		where $U: \mathcal{N}(V)^\perp \to \mathcal{N}(V)^\perp$ is an isometry.
		
		Furthermore, Furuta demonstrated in \cite{Furuta78} that hyponormality can be replaced by the weaker condition of paranormality in this context. A natural inquiry is whether this equivalence persists for broader classes within the normaloid hierarchy. The following theorem shows that this is indeed the case for absolute-$(p,r)$-paranormal operators.

		\begin{theorem}\label{thm:quasi_iso_char}
			Let $V\in\mathfrak{B}(\mathcal{H})$ be a partial isometry. The following conditions are equivalent:
			\begin{enumerate}[label=\textit{(\roman*)}]
				\item $V$ is quasinormal;
				\item $V$ is absolute-$(p,r)$-paranormal for some $p,r>0$;
				\item ${V^*}^2V^2=V^*V$.
				\item ${V^*}^2V^2\geq V^*V$.
			\end{enumerate}
		\end{theorem}
		
		\begin{proof}
			$(i)\Rightarrow(ii)$: This implication is immediate.
			
			$(ii)\Rightarrow(iii)$: Assume that $V$ is absolute-$(p,r)$-paranormal for some $p,r>0$. Let $x\in\mathcal{H}$ be an arbitrary unit vector. By Theorem \ref{thm:monotonicity}, $V$ is absolute-$(s,s)$-paranormal, where $s=\max\{p,r\}$. Using Proposition \ref{prop:def_U}, this implies that
			\begin{equation}\label{eq:U_p_r}
				\| |V|^s x \|^{2s}\leq \| |V|^s U |V|^s x \|^s,
			\end{equation}
			where $V=U|V|$ is the polar decomposition of $V$. Since $V$ is a partial isometry, we have that $U=V$. Moreover, since $|V|=V^*V$ is an orthogonal projection, we have $|V|^s = |V|$ for all $s>0$. Thus, \eqref{eq:U_p_r} reduces to
			\begin{equation*}
				\||V|x\|^2\leq \| |V|V|V|x\|,
			\end{equation*}
			which is equivalent to
			\begin{equation*}
				\|Vx\|^2\leq \|V^2x\|.
			\end{equation*}
			Since $x$ was arbitrary, we conclude that  $V$ is a paranormal partial isometry. By \cite[Theorem 1]{Ando72}, we have
			\begin{equation}\label{eq:paranormal_V}
				0\leq {V^*}^2V^2-2\lambda V^*V+\lambda^2I
			\end{equation}
			for all $\lambda>0$. Multiplying \eqref{eq:paranormal_V} by $V^*V$ on both sides, and using the fact that $V^*V\leq I$ (implying ${V^*}^2V^2 \le V^*V$), we obtain
			\begin{align*}
				0 &\leq V^*V{V^*}^2V^2V^*V-2\lambda(V^*V)^3+\lambda^2(V^*V)^2\\
				&={V^*}^2V^2-2\lambda V^*V+\lambda^2V^*V\\
				&\leq (1-\lambda)^2V^*V,
			\end{align*}
			for all $\lambda>0$. In particular, by taking $\lambda=1$, we directly get ${V^*}^2V^2=V^*V$. 
			
			$(iii)\Rightarrow(i)$: Assume that ${V^*}^2V^2=V^*V$. Multiplying from the left by $V^*$ and from the right by $V$, we get
			\begin{equation*}
				{V^*}^3V^3=V^*({V^*}^2V^2)V=V^*(V^*V)V={V^*}^2V^2=V^*V.
			\end{equation*}
			Continuing this process, we obtain ${V^*}^nV^n=V^*V$ for all $n\in\mathbb{N}$. Since $V^*V$ is an orthogonal projection, it follows that
			\begin{equation*}
				{V^*}^nV^n=(V^*V)^n
			\end{equation*}
			for all $n\in\mathbb{N}$. By \cite[p. 63]{Embry73}, we conclude that $V$ is quasinormal.
			
			$(iii)\Leftrightarrow(iv)$: This follows from the fact that $V$ is a contraction.
		\end{proof}

		As expected, the condition of paranormality cannot be replaced by normaloidness in the previous theorem, as the following example demonstrates.
		
		\begin{example}\label{ex:non_quasi_pi}
			Let 
			\[
			V=\begin{bmatrix}
				1&0&0\\
				0&0&1\\
				0&0&0
			\end{bmatrix}.
			\] 
			Then, it is easy to verify that $\|V^n\|=1=\|V\|^n$ for each $n\in\mathbb{N}$, which implies that $V$ is normaloid. Furthermore, 
			\begin{equation*}
				V^*V=\begin{bmatrix}
					1&0&0\\
					0&0&0\\
					0&0&1
				\end{bmatrix}\quad \text{ and }\quad VV^*=\begin{bmatrix}
					1&0&0\\
					0&1&0\\
					0&0&0
				\end{bmatrix}.
			\end{equation*}
			Since $V^*V$ is an orthogonal projection, it follows that $V$ is a partial isometry. On the other hand, recall that in the finite-dimensional setting, the notions of normality and quasinormality coincide. Since $V^*V\neq VV^*$, the operator $V$  is not quasinormal.
		\end{example}
		
		\begin{remark}
			Example~\ref{ex:non_quasi_pi} also illustrates that a binormal partial isometry need not be quasinormal (normal).
		\end{remark}

		Before proceeding to our next characterization, we establish a result of independent interest concerning the kernel growth of absolute-$(p,r)$-paranormal operators. Recall that the \emph{ascent} of an operator $T \in \mathfrak{B}(\mathcal{H})$, denoted by $\operatorname{asc}(T)$, is the smallest positive integer $n$ such that $\mathcal{N}(T^n) = \mathcal{N}(T^{n+1})$.

		\begin{theorem}\label{thm:pq_asc}
			Let $T\in\mathfrak{B}(\mathcal{H})$ be an absolute-$(p,r)$-paranormal operator for some $p,r>0$. Then $\operatorname{asc}(T)=1$.
		\end{theorem}
		
		\begin{proof}
			It is sufficient to show that $\mathcal{N}(T^2)\subseteq\mathcal{N}(T)$. Let $s > \max\{p,r,1\}$ and let $x \in \mathcal{N}(T^2)$ with $\norm{x}=1$. Consider the orthogonal decomposition $\mathcal{H} = \overline{\mathcal{R}(|T|^{s-1})} \oplus \mathcal{N}(|T|^{s-1})$. We can write $x = y + z$, where $y = \lim_{n\to\infty} |T|^{s-1} y_n$ for some sequence $\{y_n\}_{n\in\mathbb{N}} \subset \mathcal{H}$, and $z \in \mathcal{N}(|T|^{s-1}) = \mathcal{N}(T)$.
			
			Since $z \in \mathcal{N}(T)$, we have $T^2 z = 0$, and thus
			\begin{equation*}
				0 = T^2x = T^2y =  \lim_{n\to\infty} T^2 |T|^{s-1} y_n.
			\end{equation*}
			Using the identity from \eqref{eq:trans_equiv},   we observe that
			\begin{equation*}
				\norm{|T|^s U|T|^s y_n} = \norm{|T^*|^{s-1} T^2 |T|^{s-1} y_n} \to 0 \quad \text{as } n \to \infty.
			\end{equation*}
			Next, by Theorem \ref{thm:monotonicity}, $T$ is absolute-$(s,s)$-paranormal. The defining inequality \eqref{eq:power_s} then implies
			\begin{equation*}
				\norm{|T|^s y_n}^2 \leq \norm{|T|^s U |T|^s y_n}
			\end{equation*}
			so $|T|^s y_n \to 0$ as $n \to \infty$.
			
			To conclude, notice that  
			\begin{equation*}
				|T|y = |T| \left( \lim_{n\to\infty} |T|^{s-1} y_n \right) =  \lim_{n\to\infty} |T|^s y_n = 0.
			\end{equation*}
			Thus, $y \in \mathcal{N}(|T|)$. However, by construction, $y \in \overline{\mathcal{R}(|T|^{s-1})} = \overline{\mathcal{R}(|T|)}$. Since $\overline{\mathcal{R}(|T|)} \cap \mathcal{N}(|T|) = \{0\}$, it follows that $y = 0$. Consequently, $x = z \in \mathcal{N}(T)$, which proves $\mathcal{N}(T^2) = \mathcal{N}(T)$, i.e.,  $\operatorname{asc}(T)=1$.
		\end{proof}
		
		\begin{theorem}\label{thm:pq_paranormal_root_iso}
			Let $T\in\mathfrak{B}(\mathcal{H})$. The following conditions are equivalent:
			\begin{enumerate}[label=\textit{(\roman*)}]
				\item $T$ is a quasinormal partial isometry;
				\item $T$ is an absolute-$(p,r)$-paranormal operator for some $p,r>0$ and $T^n$ is a partial isometry for some $n\in\mathbb{N}$.
			\end{enumerate} 
		\end{theorem}
		\begin{proof}
		$(i)\Rightarrow(ii)$: This follows from \eqref{eq:iso_plus_zero}.
		
		$(ii)\Rightarrow(i)$: Assume that $T$ is absolute-$(p,r)$-paranormal and $T^n$ is a partial isometry for some $n \in \mathbb{N}$. By Theorem~\ref{thm:pr_normaloid}, $T$ is normaloid, which implies
		\begin{equation*}
			\|T\|^n = \|T^n\| \leq 1,
		\end{equation*}
		showing that $T$ is a contraction. Furthermore, Theorem~\ref{thm:pq_asc} establishes that $\mathcal{N}(T) = \mathcal{N}(T^2)$. By \cite[Theorem A]{Gupta80} (cf. \cite[Theorem 4]{EzzahraouietAl18}), we deduce that $T$ is a partial isometry. Finally, Theorem \ref{thm:quasi_iso_char} yields the quasinormality of $T$.
		\end{proof}
		
		\begin{remark}
			The absolute-$(p,r)$-paranormality in Theorem \ref{thm:pq_paranormal_root_iso} can be replaced by $k$-paranormality, as both conditions effectively control both the norm and the kernel growth (ascent). This control is crucial; the normaloid property alone is insufficient to ensure that $T$ is a partial isometry when its power is. For instance, the matrix
			\begin{equation*}
				T = \begin{bmatrix} 1 & 0 & 0 \\ 0 & 0 & 1/2 \\ 0 & 0 & 0 \end{bmatrix}
			\end{equation*}
			is normaloid and $T^2$ is a partial isometry, yet $T$ is not, as 
			$$T^*T = \begin{bmatrix}
				
				1&0&0\\
				
				0&0&0\\
				
				0&0&\frac{1}{4}
				
			\end{bmatrix}$$
			 is not an orthogonal projection. 
			 
			\medskip 
			
			Similarly, neither binormality nor posinormality is sufficient. The $2 \times 2$ nilpotent matrix $T = \begin{bmatrix} 0 & 2 \\ 0 & 0 \end{bmatrix}$ is binormal and its square is a partial isometry, but $T$ is not. For the posinormal case, consider the matrix
			$
				T = \begin{bmatrix} 1 & 1 \\ 0 & -1 \end{bmatrix}.
			$
			Since $T$ is invertible, it is posinormal by \cite[Theorem 3.1]{Rhaly94}. While $T^2 = I$ is a partial isometry, $T$ itself is not, as $T^*T = \begin{bmatrix} 1 & 1 \\ 1 & 2 \end{bmatrix}$ is not a projection.
		\end{remark}

		\bigskip 
		\section*{Declarations} 	
		
		\vspace{0.5cm}
		

        \noindent{\bf{Funding}}\\
		This work has been supported by the Ministry of Science, Technological Development and Innovation of the Republic of Serbia [Grant Number: 451-03-34/2026-03/200102].	
        \vspace{0.5cm}
        
		\noindent{\bf{Availability of data and materials}}\\
		\noindent No data were used to support this study.
		\vspace{0.5cm}\\
		\noindent{\bf{Competing interests}}\\
		\noindent The authors declare that they have no competing interests.
		\vspace{0.5cm}
		
		\noindent{\bf{Author contribution}}\\
		\noindent
		The work was a collaborative effort of all authors, who contributed equally to writing the article. All authors have read and approved the final manuscript.
		
		\vspace{0.5cm}
		
		
		

\begin{thebibliography}{10} 
			\footnotesize
			
			\bibitem{Aluthge90}
			A. Aluthge. \emph{On $p$-hyponormal operators for $0<p<1$}, Integral Equations Operator Theory, \textbf{13} (1990), 307--315.
			
			\bibitem{Aluthge96}
			A. Aluthge. \emph{Some generalized theorems on p-hyponormal operators}, Integral Equations Operator Theory, \textbf{24} (1996), 497--501.
			
			\bibitem{Ando72}
			T. Ando. \emph{Operators with a norm condition}, Acta Sci. Math. (Szeged), \textbf{33} (1972), 169--172.
			
			\bibitem{Brown53}
			A. Brown. \emph{On a class of operators}, Proc. Amer. Math. Soc.. \textbf{4} (1953), 723--728. 
			
			\bibitem{Campbell72}
			S. L. Campbell. \emph{Linear operators for which $T^*T$ and $TT^*$ commute}, Proc. Am. Math. Soc. \textbf{34} (1972), 177--180.
			
			
			\bibitem{Campbell75}
			S. L. Campbell. \emph{Linear operators for which $T^*T$ and $TT^*$ commute II}, Pacific Journal of Mathematics, \textbf{53} (1980), 355--361.
			
			\bibitem{Campbell80}
			S. L. Campbell. \emph{Linear operators for which $T^*T$ and $TT^*$ commute III}, Pacific Journal of Mathematics, \textbf{91}(1) (1980), 39--45. 
			
			\bibitem{Conway91}
			J. B. Conway. \emph{The Theory of Subnormal Operators}, Math. Surveys Monographs, vol. 36. Amer. Math. Soc. Providence, (1991).
			
			\bibitem{DuggalKubrusly11}
			B. P. Duggal, C. S. Kubrusly. \emph{Quasi-similar $k$-paranormal operators}, Operators and Matrices, \textbf{5}(3) (2011), 417--423.
			
			\bibitem{Embry73}
			M. R. Embry. \emph{A generalization of the Halmos-Bram criterion for subnormality}, Acta. Sci. Math. (Szeged), \textbf{35} (1973), 61--64.
			
			
			\bibitem{EzzahraouietAl18}
			H. Ezzahraoui, M. Mbekhta, A. Salhi, E. H. Zerouali. \emph{A note on roots and powers of partial isometries}. Arch. Math., \textbf{110}  (2018), 251--259, \url{https://doi.org/10.1007/s00013-017-1116-2}.
			
			\bibitem{FujiiJungLee00}
			M. Fujii, D. Jung, S. H. Lee, Y. Lee, R. Nakamoto. \emph{Some classes of operators related to paranormal and $\log$-hyponormal operators}, Math Japonica \textbf{51}(3) (2000), 395--402. 
			
			\bibitem{Furuta78}
			T. Furuta. \emph{Partial isometries and similar operators}, Revue Roumaine Math. Pure et Appl.,
			\textbf{8} (1978), 1157--1166.
			
			\bibitem{Furuta01}	
			T. Furuta. \emph{Invitation to Linear Operators. From Matrices to Bounded Linear Operators on a Hilbert Space}, Taylor and Francis, London and New York, (2001).
			
			\bibitem{FurutaItoYamazaki98}
			T. Furuta, M. Ito, T. Yamazaki. \emph{A subclass of paranormal operators including class of log-hyponormal and several related
				classes}, Scientiae Mathematicae, \textbf{1} (1998), 389--403. 
				
			\bibitem{Gupta80}
			B. C. Gupta. \emph{A note on partial isometries  II}, Indian J. Pure Appl. Math., \textbf{11}
			(1980), 208--211.
			
			\bibitem{Halmos50}
			P. R. Halmos. \emph{Normal dilations and extensions of operators}, Summa Bras. Math., \textbf{2} (1950), 125--134. 
			
			\bibitem{Istratescu67}
			V. Istr\v a\cb{t}escu. \emph{On some hyponormal operators}, Pacific J. Math. \textbf{22} (1967), 413--417.
			
			\bibitem{IstratescuIstratescu67}
			I. Istr\v a\cb{t}escu, V. Istr\v a\cb{t}escu. \emph{On Some Classes of Operators. $I$}, Proc. Japan Acad., \textbf{43} (1967), 605--606.
			
			\bibitem{IstratescuSaitoYoshino66}
			V. Istr\v a\cb{t}escu, T. Saito, T. Yoshino. \emph{On a class of operators}, Tohoku Math. J., \textbf{18} (1966), 410--413. 
			
			\bibitem{Kubrusly03}
			C. S. Kubrusly. \emph{Hilbert Space Operators}, Birkh\" auser, Boston, (2003).
			
			\bibitem{KubruslyDuggal07}
			C. S. Kubrusly, B. P. Duggal. \emph{On posinormal operators},
			Adv. Math. Sci. Appl. \textbf{17} (2007), 131--148.
			
			\bibitem{KubruslyDuggal10}
			C. S. Kubrusly, B. P. Duggal. \emph{A note on $k$-paranormal operators}, Operators and Matrices, \textbf{4}(2) (2010), 213--223.
			
			
			\bibitem{KubruslyVieieraZanni16}
			C. S. Kubrusly, P. C. M. Vieira, J. Zanni.  \emph{Powers of posinormal operators},
			Operators and Matrices, \textbf{10}(1) (2016), 15--27. 
			
			\bibitem{Nielsen80}
			O. A. Nielsen. \emph{Direct Integral Theory}, CRC Press, (1980).
			
			\bibitem{McCarthy67}
			C. A. McCarthy, $C_p$, Israel J. Math., \textbf{5} (1967), 249--271.
			
			\bibitem{Rhaly94}
			H. C. Rhaly, Jr. \emph{Posinormal operators},
			J. Math. Soc. Japan \textbf{46}(4) (1994), 587--605. 
			
			\bibitem{Stampfli62}
			J. G. Stampfli. \emph{Hyponormal operators}, Pacific J. Math., \textbf{12} (1962), 1453--1458.
			
			\bibitem{Stankovic24}
			H. Stankovi\' c. \emph{Conditions implying self-adjointness and normality of operators}, Complex Analysis and Operator Theory, \textbf{18}(149) (2024), \url{https://doi.org/10.1007/s11785-024-01596-0}.
			
			\bibitem{Stankovic25}
			H. Stankovi\' c. \emph{On Fong-Tsui conjecture and binormality of operators}, Mathematica Slovaca, \textbf{75}(5) (2025), 1221--1228,
			\url{https://doi.org/10.1515/ms-2025-0088}.
			
			\bibitem{StankovicKubrusly25}
			H. Stankovi\'c, C. Kubrusly. \emph{On roots of normal operators and extensions of Ando's Theorem},  Annals of Functional Analysis, \textbf{16}(60) (2025),
			\url{https://doi.org/10.1007/s43034-025-00455-z}.
 
			\bibitem{Xia83}
			D. Xia. \emph{Spectral Theory of Hyponormal Operators}, Birkh\" auser Verlag, Boston, (1983).
			
			\bibitem{YamazakiYanagida00}
			T. Yamazaki, M. Yanagida. \emph{A further generalization of paranormal operators}, Scientiae Mathematicae,   \textbf{3}(1) (2000), 23--31.
			
			\bibitem{YamazakiYanagida03}
			T. Yamazaki, M. Yanagida. \emph{Relations between two operator inequalities and their applications to paranormal operators}, Acta Scientiarum Mathematicarum, \textbf{69} (2003), 377--389. 
			
		\end{thebibliography}
	\end{document}